\documentclass[12pt]{amsart}
\usepackage{a4}
\usepackage{amsthm, amsfonts, amssymb,latexsym}
\usepackage{enumerate, color, comment}
\usepackage[english]{babel}
\usepackage[latin1]{inputenc}
\usepackage{fixme}
\usepackage[backref=page]{hyperref}

\theoremstyle{plain}
\newtheorem{theorem}{Theorem}
\newtheorem{lemma}[theorem]{Lemma}
\newtheorem{cor}[theorem]{Corollary}

\newcommand{\D}{\Delta}

\numberwithin{equation}{section}

\newcommand{\beq}[1]{\begin{equation}\label{#1}}
\newcommand{\eeq}{\end{equation}}

\title[Variations on the sum-product problem II]{Variations on the sum-product problem II}

\author[B. Murphy, O. Roche-Newton and I. Shkredov]{Brendan Murphy, Oliver Roche-Newton and Ilya D. Shkredov}

\address{B. Murphy: School of Mathematics, University Walk, Bristol, UK, BS8 1TW\\ 
and Heilbronn Institute of Mathematical Research}
\email{brendan.murphy@bristol.ac.uk  }

\address{O. Roche-Newton: 69 Altenberger Stra{\ss}e, Johannes Kepler Universit\"{a}t, Linz, Austria }
\email{o.rochenewton@gmail.com }

\address{I. D. Shkredov: Steklov Mathematical Institute, ul. Gubkina, 8, Moscow, Russia, 119991\\
IITP RAS,  
Bolshoy Karetny per. 19, Moscow, Russia, 127994\\
and 
MIPT, 
Institutskii per. 9, Dolgoprudnii, Russia, 141701
 }
\email {ilya.shkredov@gmail.com}

\begin{document}

\begin{abstract}
This paper is a sequel to a paper entitled \textit{Variations on the sum-product problem} by the same authors \cite{MRNS}. In this sequel, we quantitatively improve several of the main results of \cite{MRNS}, as well as generalising a method from \cite{MRNS} to give a near-optimal bound for a new expander. 

The main new results are the following bounds, which hold for any finite set $A \subset \mathbb R$:
\begin{align*}
\exists a \in A \text{ such that }|A(A+a)| &\gtrsim |A|^{\frac{3}{2}+\frac{1}{186}},
\\|A(A-A)| &\gtrsim |A|^{\frac{3}{2}+\frac{1}{34}},
\\|A(A+A)| &\gtrsim |A|^{\frac{3}{2}+\frac{5}{242}},
\\ |\{(a_1+a_2+a_3+a_4)^2+\log a_5 : a_i \in A \}| &\gg \frac{|A|^2}{\log |A|}.
\end{align*}
\end{abstract} 
\maketitle
\section{Introduction}

Throughout this paper, the standard notation
$\ll,\gg$ is applied to positive quantities in the usual way. Saying $X\gg Y$ means that $X\geq cY$, for some absolute constant $c>0$. The notation $X \approx Y$ denotes that $X \gg Y$ and $X \ll Y$ occur simultaneously. All logarithms in the paper are base $2$. We use the symbols $\lesssim, \gtrsim$ to suppress both constant and logarithmic factors. To be precise, we write $X \gtrsim Y$ if there is some absolute constant $c>0$ such that $X \gg Y/(\log X)^c$.

This paper considers several variations on the sum-product problem, all of which follow a common theme. The story of the sum-product problem begins with the Erd\H{o}s-Szemer\'{e}di conjecture, which states that for any finite $A \subset \mathbb Z$ and for all $\epsilon>0$,
$$\max \{ |A+A|,|AA| \} \geq c_{\epsilon} |A|^{2-\epsilon},$$
where
$$A+A:=\{a+b :a,b \in A \}$$
is the \textit{sum set} and
$$AA:=\{ab: a,b \in A\}$$
is the \textit{product set}. Although the conjecture was originally stated for subsets of the integers, it is also widely believed to be true over the reals. 

Modern literature on the sum-product problem often focuses on the problem of proving that certain sets defined by a combination of different arithmetic operations on elements of a set $A$ are always significantly larger than $|A|$. Growth results of this type are often referred to as \textit{expanders}.

In \cite{MRNS}, the authors considered several expander problems. The aim of this sequel is to improve the main results from \cite{MRNS}. 

One result that was established in \cite{MRNS} is that, for any $A \subset \mathbb R$,
\begin{equation}
|A(A+A)| \gtrsim |A|^{\frac{3}{2}+\frac{1}{178}}.
\label{mainold}
\end{equation}
This result gave a small quantitative improvement on the inequality $|A(A+A)| \gg |A|^{3/2}$, which follows from a simple application of the Szemer\'{e}di-Trotter Theorem (see \cite[Lemma 3.2]{RNZ} for a formal proof). 

The exponent $3/2$ arises often as a threshold for what is achievable in sum-product type problems by using the Szemer\'{e}di-Trotter Theorem in the simplest way that we know of, essentially copying the arguments of Elekes \cite{E}. To improve on these fundamental results, refinements of these techniques have been required. See for example, \cite{LRN} and \cite{SS}. Most of these basic sum-product results with exponent $3/2$ have by now been improved in the Euclidean setting, but some remain out of reach. For example, a simple application of the Szemer\'{e}di-Trotter Theorem (see for example Exercise 8.3.3 in \cite{TV}) yields the bound $|AA+A| \gg |A|^{3/2}$, and no improvement is known.

The main new theorem that we prove in this paper is another example of a result that breaks the $3/2$ threshold.

\begin{theorem} \label{thm:main1} Let $A \subset \mathbb R$ be finite. Then, there exists $a \in A$ such that
$$|A(A+a)| \gtrsim |A|^{\frac{3}{2}+\frac{1}{186}} .$$
\end{theorem}

This improves on \cite[Theorem 2.9]{MRNS}, in which it was established that there is some $a \in A$ such that $|A(A+a)| \gg |A|^{3/2}.$

\subsection{Sketch of the proof of Theorem \ref{thm:main1}}

There are two main new lemmas which go into the proof of Theorem \ref{thm:main1}. The first of these is a lemma which states that there exists $a \in A$ such that $|A(A+a)| \geq |A|^{\frac{3}{2}+c}$, where $c>0$ is an absolute constant, unless the multiplicative energy\footnote{The multiplicative energy of $A$ is the number of solutions to the equation 
$$a_1a_2=a_3a_4,\,\,\,\,\,\,\,(a_1,a_2,a_3,a_4) \in A.$$
See section \ref{sec:notation} for more on the multiplicative energy and other types of energy.} of $A$ is essentially as large as possible. This lemma is proved using the Szemer\'{e}di-Trotter Theorem. In proving such a result, we improve qualitatively and quantitatively on some of the main lemmas from \cite{MRNS}. See the forthcoming Lemma \ref{pinned}.

The second new lemma (see the forthcoming Lemma \ref{Sh4} for a more precise statement) proves that if the product set of $A$ is very small then the bound $|A(A+\alpha)| \gg |A|^{\frac{3}{2}+c}$ holds for any non-zero $\alpha \in \mathbb R$. The proof of this is a little more involved, using techniques from additive combinatorics, and is closely related to the work of the third author in \cite{Sh4}. A non-trivial result bounding the additive energy for sets with small product set, also due to the third author in \cite{Sh3}, is used as a black box in the proof of this lemma. See the forthcoming Theorem \ref{hardthm} for the statement.

The Balog-Szemer\'{e}di-Gowers Theorem tells us that if a set $A$
has large multiplicative energy, then there is a large subset $A' \subset A$
such that $|A' A'|$
is small.  Therefore, one can then use the Balog-Szemer\'{e}di-Gowers Theorem and combine these two lemmas together to conclude the proof. However, we instead use a technical argument, building on the work of Konyagin and Shkredov \cite{KS}, which allows us to avoid an application of the Balog-Szemer\'{e}di-Gowers Theorem and thus to improve the aforementioned constant $c$.

\subsection{Further new results}

The proof of inequality \eqref{mainold} in \cite{MRNS} followed a similar structure to the above sketch, and the Balog-Szemer\'{e}di-Gowers Theorem was used to conclude the argument. Once again, we are able to make this argument more efficient by using tools from \cite{KS} to avoid using the Balog-Szemer\'{e}di-Gowers Theorem, resulting in the following two results.

\begin{theorem} \label{thm:main2} Let $A \subset \mathbb R$ be finite. Then,
$$|A(A+A)| \gtrsim |A|^{\frac{3}{2}+\frac{5}{242}} .$$
\end{theorem}

\begin{theorem} \label{thm:main3} Let $A \subset \mathbb R$ be finite. Then,
$$|A(A-A)| \gtrsim |A|^{\frac{3}{2}+\frac{1}{34}} .$$
\end{theorem}

Theorem \ref{thm:main2} gives an improvement on \eqref{mainold}, while Theorem \ref{thm:main3} gives an improvement on equation (49) in \cite{MRNS}.

Another of the main results in \cite{MRNS} was the bound
\begin{equation}
|A(A+A+A+A)| \gg \frac{|A|^2}{\log |A|}.
\label{mainold2}
\end{equation}
Note that \eqref{mainold2} is optimal, up to finding the correct power of the logarithmic factor, as can be seen by taking $A=\{1,2,\dots,N\}$. In the last of our new theorems, we follow a similar argument to prove the following, admittedly curious, expander bound.

\begin{theorem} \label{thm:main4} Let $A \subset \mathbb R^+$ be finite. Then,
$$|\{(a_1+a_2+a_3+a_4)^2+\log a_5 : a_i \in A \}| \gg \frac{|A|^2}{\log |A|} .$$
\end{theorem}
Note that the simple example whereby $A=\{1,2,\dots,N\}$ illustrates that Theorem \ref{thm:main4} is also optimal up to the logarithmic factor.

\subsection{The structure of the rest of this paper}

The rest of the paper will be structured as follows. Section \ref{sec:notation} will be used to introduce some notation and preliminary results that will be used throughout the paper. As mentioned above, there are two main new lemmas in this paper. Section \ref{sec:lemma1} is devoted to proving the first of these, and section \ref{sec:lemma2} the second. Section \ref{sec:thm1} is used to conclude the proof of Theorem \ref{thm:main1}. In section \ref{sec:thm2and3} the proofs of Theorem \ref{thm:main2} and \ref{thm:main3} are concluded. Section \ref{sec:5var} is devoted to the proof of Theorem \ref{thm:main4}.

\subsection{A note on an earlier preprint \cite{ORN}}
This paper supercedes the preprint \cite{ORN} by the second author.

\section{Notation and preliminary results} \label{sec:notation}
Given finite sets $A,B \subset \mathbb R$, the additive energy of $A$ and $B$ is the number of solutions to
the equation
$$a_1- b_1 = a_2-b_2$$
The additive energy is denoted $E^+(A,B)$. Let
$$r_{A-B}(x):=|\{(a,b) \in A \times B : a-b=x \}|.$$
Note that $r_{A-B}(x)=|A \cap (B+x)|$. The notation of the representation function $r$ will be used
with flexibility throughout this paper, with the information about the kind of representations
it counts being contained in a subscript. For example,
$$r_{(A-A)^2+(A-A)^2}(x)=|\{(a_1,a_2,a_3,a_4) \in A^4 : (a_1-a_2)^2+(a_3-a_4)^2=x \}|.$$
Note that
$$E^+(A,B)=\sum_{ x \in A-B } r_{A-B}^2(x).$$
The shorthand $E^+(A)=E^+(A,A)$ is used. 

The notion of energy can be extended to an
arbitrary power $k$. We define $E_k^+(A)$ by the formula
$$E_k^+(A)= \sum_{x \in A-A} r_{A-A}^k(x).$$
Similarly, the \textit{multiplicative energy} of $A$ and $B$, denoted $E^{\times}(A,B)$, is the number of solutions to the equation
$$\frac{a_1}{b_1}=\frac{a_2}{b_2}$$
such that $a_1,a_2 \in A$ and $b_1,b_2 \in B$.
For $x\neq 0$, let $A_x$ denote the set $A_x=A \cap x^{-1}A$ and note that $r_{A/A}(x)=|A_x|$.

A simple but important feature of energies is that the Cauchy-Schwarz inequality can be used to convert an upper bound for energy into a lower bound for the cardinality of a set. In particular, it follows from the Cauchy-Schwarz inequality that
\begin{equation}
E^{\times}(A,B) \geq \frac{|A|^2|B|^2}{|AB|}.
\label{CSbound}
\end{equation}

The notions of additive and multiplicative energy have been central in the literature on
sum-product estimates. For example, the key ingredient in the beautiful work of Solymosi \cite{So}, which until recently held the record for the best known sum-product estimate, is the
following bound:
\begin{theorem} \label{thm:soly}
For any finite $A \subset \mathbb R$,
$$E^{\times}(A) \ll |A+A|^2\log |A|.$$
\end{theorem}

A major tool that is used in this paper several times, both explicitly and implicitly, is the Szemer\'{e}di-Trotter Theorem. In particular, we will need the following result, which follows from a simple application of the Szemer\'{e}di-Trotter Theorem. See, for example, Corollary 8.8 in \cite{TV}.

\begin{lemma} \label{collinear}
Let $A \subset \mathbb R$ be a finite set. Then there are $O(|A|^4\log|A|)$ collinear triples in $A \times A$.
\end{lemma}

In a recent paper of Konyagin and Shkredov \cite{KS}, a new characteristic for a finite set of
real numbers $A$ was considered. Define $d_*(A)$ by the formula
$$d_*(A)=\min_{t>0} \min_{\emptyset \neq Q, R \subset \mathbb R \setminus \{0\}} \frac{|Q|^2|R|^2}{|A|t^3},$$
where the second minimum is taken over all $Q$ and $R$ such that $\max \{|Q|,|R|\} \geq |A|$ and such that for every $a \in A$, the bound $|Q \cap aR^{-1}| \geq t$ holds. Konyagin and Shkredov proved the following lemma:
\begin{lemma}[Lemma 13, \cite{KS}] \label{thm:KSlem} For any $A,B \subset \mathbb R$ and any $\tau \geq 1 $
$$|\{x :r_{A-B}(x) \geq \tau \}| \ll \frac{|A||B|^2}{\tau^3} d_*(A).$$
\end{lemma}
To put this in the language introduced in \cite{Sh2} and also used in \cite{KS}, this lemma states that every set $A$ is a Szemer\'{e}di-Trotter set with parameter $O(d_*(A))$. The main theoretical tool in the proof of Lemma \ref{thm:KSlem} is the Szemer\'{e}di-Trotter Theorem. Lemma \ref{thm:KSlem} generalises an earlier result in which the bound
\begin{equation}
|\{x :r_{A-B}(x) \geq \tau \}| \ll \frac{|A||B|^2}{\tau^3} d(A)
\label{KSold}
\end{equation}
was established, where $d(A)=\min_{C \neq \emptyset} \frac{|AC|^2}{|A||C|}$. See \cite[Lemma 7]{RRNS} for a proof. As pointed out in \cite{KS}, $d_*(A) \leq d(A)$, since for any non empty $C$ we can take $t=|C|$, $Q=AC$ and $R=C^{-1}$ in the definition of $d_*(A)$.

\section{A bound on sums of multiplicative energies with shifts} \label{sec:lemma1}

The aim of this section is to prove the first of the two main new lemmas of this paper. This is the following lemma, which gives an improvement of Lemma 2.4 in \cite{MRNS}, in the case when the sets involved are approximately the same size, unless the multiplicative energy is essentially as large as possible.

\begin{lemma} \label{pinned1} Let $ A,B,C \subset \mathbb R$ be a finite sets such that $|A| \approx |C|$. Then
$$\sum_{a \in A} E^{\times}(B,C-a) \ll E^{\times}(B)^{1/2}|A|^2\log^{1/2}|A|+|A|^3+|A||B|^2.$$
\end{lemma}

\begin{proof} 
We have
\begin{align*}
\sum_{a \in A} E^{\times}(B,C-a)&=|\{(a,b,b',c,c') \in A \times B \times B \times C \times C : b(c-a)=b'(c'-a)\}|
\\& \leq |\{(a,b,b',c,c') \in A \times B \times B \times C \times C : b(c-a)=b'(c'-a) \neq 0 \}| 
\\& +|A|^3 +|A|^2|B| + |A||B|^2
\\& \leq \left|\left\{(a,b,b',c,c') \in A \times B \times B \times C \times C , : \frac{b}{b'}=\frac{c'-a}{c-a} \neq 0\right\}\right |
\\ &+|A|^3 +E^{\times}(B)^{1/2}|A|^2 + |A||B|^2.
\end{align*}
The remaining task is to bound the main term. Applying the Cauchy-Schwarz inequality yields
\begin{align} \label{eq1}
&\left|\left\{(a,b,b',c,c') \in A \times B \times B \times C \times C  : \frac{b}{b'}=\frac{c'-a}{c-a} \neq 0\right\}\right| \nonumber
\\& =\sum_{x \neq 0} r_{B/B}(x)n(x) \nonumber
\\& \leq \left(\sum_x r_{B/B}^2(x)\right)^{1/2} \left(\sum_{x\neq 0} n^2(x) \right)^{1/2} \nonumber
\\& = E^{\times}(B)^{1/2}  \left(\sum_{x\neq 0} n^2(x) \right)^{1/2},
\end{align}
where
$$n(x)=\left| \left \{ (a,c,c') \in A\times C \times C: x=\frac{c'-a}{c-a} \right\} \right |.$$
Note that
\begin{align*}
\sum_x n^2(x)&\leq \left | \left \{(a_1,a_2,c_1,c_2,c_1',c_2') :\frac{c_1'-a_1}{c_1-a_1}=\frac{c_2'-a_2}{c_2-a_2} \right \} \right|
\\&=\left | \left \{(a_1,a_2,c_1,c_2,c_1',c_2') :\frac{c_2-a_2}{c_1-a_1}=\frac{c_2'-a_2}{c_1'-a_1} \right \} \right|.
\end{align*}
The identity $\frac{c_2-a_2}{c_1-a_1}=\frac{c_2'-a_2}{c_1'-a_1}$ occurs only if the three points $(a_1,a_2),(c_1,c_2),(c_1',c_2') \in (A\cup C) \times (A \cup C)$ are collinear. By Lemma \ref{collinear}, there are $O(|A\cup C|^4 \log|A|)$ such collinear triples, and so
$$\sum_x n^2(x) \ll |A|^4\log|A|.$$
Combining this with \eqref{eq1}, we have
$$\sum_{a \in A} E^{\times}(B,C-a) \ll E^{\times}(B)^{1/2}|A|^2\log^{1/2}|A|+|A|^3+|A||B|^2.$$

\end{proof}

\begin{cor} \label{pinned} Let $A \subset \mathbb R$  be finite. Then there exists $a \in A$ such that
\begin{equation}
E^{\times}(A)|A(A+a)|^2 \gg \frac{|A|^6}{\log |A|}.
\label{pinned+}
\end{equation}
Similarly, there exists $b \in A$ such that
\begin{equation}
E^{\times}(A)|A(A-b)|^2 \gg \frac{|A|^6}{\log |A|}.
\label{pinned-}
\end{equation}
\end{cor}

\begin{proof}
By Lemma \ref{pinned1}, we have
$$\sum_{a \in A} E^{\times}(A,A+a) \ll E^{\times}(A)^{1/2}|A|^2 \log^{1/2}|A| +  |A|^3 \ll E^{\times}(A)^{1/2}|A|^2 \log^{1/2}|A|.$$
Therefore, by the pigeonhole principle, there exists $a \in A$ such that 
$$E^{\times}(A,A+a) \ll E^{\times}(A)^{1/2}|A|\log^{1/2}|A|.$$ By the Cauchy-Schwarz inequality
$$\frac{|A|^4}{|A(A+a)|} \leq E^{\times}(A,A+a) \ll E^{\times}(A)^{1/2}|A|\log^{1/2}|A|.$$
A rearrangement of this inequality completes the proof. The proof of \eqref{pinned-} is essentially identical.
\end{proof}

Corollary \ref{pinned} provides a pinned version of the following result, which was the main lemma from \cite{MRNS}. 

\begin{lemma} \label{thm:oldmainlemma}
For any finite sets $A,B,C \in \mathbb R$,
$$E^{\times}(A)|A(B+C)|^2 \gg \frac{|A|^4|B||C|}{\log |A|}.$$
\end{lemma}


\section{A conditional lower bound on $|A(A+\alpha)|$} \label{sec:lemma2}

The second main new lemma of this paper is the following.


\begin{lemma} \label{Sh4} Let $A \subset \mathbb R$ and $\alpha \in \mathbb R \setminus \{ 0 \}$. Then
$$|A(A+\alpha)| \gtrsim \frac{E^{\times}(A)^2}{|A|^{\frac{58}{13}} d_*^{\frac{7}{13}}(A)}.$$
\end{lemma}

In particular, this result implies the following statement:
$$|AA| \leq M|A| \Rightarrow |A(A+\alpha)| \gg_M |A|^{\frac{3}{2}+c}.$$
Indeed, since by the Cauchy-Schwarz inequality $E^{\times}(A) \geq |A|^3/M$ and since $d_*(A) \leq d(A) \leq M^2$, we have
$$|A(A+\alpha)| \gtrsim \frac{|A|^{\frac{20}{13}}}{M^{\frac{40}{13}}}.$$

In the proof of Lemma \ref{Sh4} we will need the following simple lemma. See \cite[Lemma 4]{Sh4}.

\begin{lemma}
    Let $G$ be an abelian group and $A\subset G$ be a finite set.
    Then there is $z$ such that
    $$
        \sum_{x\in zA} |(zA) \cap x (zA)| \gg \frac{E^{\times} (A)}{|A| \log |A|} \,.
    $$
\label{l:sigma&E}
\end{lemma}

We will also need the following result of the third author \cite{Sh3}, which tells us that a set $A$ such that the parameter $d_*(A)$ is small has small additive energy. One can obtain similar but weaker results using the Szemer\'{e}di-Trotter Theorem in a more elementary way, see for example Corollary 8 in \cite{RRNS}, but the important thing for our application of this result is that the exponent $32/13$ is smaller than $5/2$.

\begin{theorem}[\cite{Sh2}, Theorem 5.4] \label{hardthm} Let $A\subset \mathbb R$ be a finite set.
Then
$$E^+(A) \ll d_* (A)^{7/13} |A|^{32/13} \log^{71/65}|A|.$$
\end{theorem}

The result in \cite{Sh3} did not use the quantity $d_* (A)$ because it was introduced later in \cite{KS2} but one can check that the arguments of \cite{Sh3} work for any Szemer\'{e}di-Trotter set. The only fact we need is an upper bound for size of a set $\{ x ~:~r_{A-B} (x) \ge \tau \}$, so Lemma \ref{thm:KSlem} is enough for us.


\begin{proof}[Proof of Lemma \ref{Sh4}] Without loss of generality, we may assume that $0 \notin A$.
Applying Lemma \ref{l:sigma&E} and writing $B=zA$, 
we have 
\begin{equation}
\sum_{\lambda \in B}|B \cap \lambda^{-1} B| \gg \frac{E^{\times}(A)}{|A| \log |A|}.
\label{sumbound}
\end{equation}

Next we double count the number of solution to the equation
\begin{equation}
b_1(b_1'+\alpha z)=b_2(b_2'+\alpha z)
\label{energy}
\end{equation}
such that $b_1,b_2 \in B$, $b_1' \in B_{b_1}$ and $b_2' \in B_{b_2}$.

Let $S$ denote the number of solutions to \eqref{energy}. Suppose that we have such a solution. Then $b_1'=b_1''/b_1$ and $b_2'=b_2''/b_2$ for some $b_1'',b_2'' \in B$. Therefore
$$\alpha z b_1 +b_1''=\alpha z b_2+b_2'',$$
and it follows that $S \leq E^+(B,\alpha z B)$. An application of the Cauchy-Schwarz inequality then gives $E^+(B,\alpha z B) \leq E^+(B)^{1/2}E^+(\alpha z B)^{1/2}=E^+(B)=E^+(A)$. So, by Theorem \ref{hardthm}
\begin{equation}
\label{upperbound}
S \leq E^+(A) \lesssim d_* (A)^{7/13} |A|^{32/13}.
\end{equation}

On the other hand denote
$$n(t)=|\{(b,b') \in B \times B_b : b(b'+\alpha z)=t\}|$$
and  note that $S=\sum_{t} n^2(t)$. Also, $n(t)>0$ only if $t \in B(B+\alpha z)$. Then, by \eqref{sumbound}, the Cauchy-Schwarz inequality and \eqref{upperbound}
\begin{align}
\label{final'}
\frac{E^{\times}(A)^2}{|A|^2 \log^2 |A|}& \ll\left( \sum_{\lambda \in B}|B \cap \lambda^{-1} B| \right)^2
\\& = \left(\sum_t n(t)\right)^2
\\& \leq |B(B+\alpha z)| \sum_t n^2(t)
\\ & \lesssim   |B(B+\alpha z)|  d_* (A)^{7/13} |A|^{32/13}.
\label{final'}
\end{align}

Finally, note that $|B(B+\alpha z)|=|A(A+\alpha)|$. We conclude that
$$|A(A+\alpha)| \gtrsim \frac{E^{\times}(A)^2}{|A|^{\frac{58}{13}}d_*(A)^{\frac{7}{13}}}.$$
\end{proof}

Actually, one can see that we have proved the inequality $|A|^2 |A(A+\alpha)| E^{+} (A) \gtrsim (E^{\times} (A))^2$ for any finite subset of reals.

\section{Proof of Theorem \ref{thm:main1}} \label{sec:thm1}

Before proving Theorem \ref{thm:main1}, we need one more lemma.

\begin{lemma} \label{thm:bigenergy} Let $A \subset \mathbb R$ and suppose that $E^{\times}(A) \geq \frac{|A|^3}{K}$. Then for any $\alpha \in \mathbb R \setminus \{0\}$
$$|A(A+\alpha)| \gtrsim \frac{|A|^{20/13}}{K^{40/13}}.$$
\end{lemma}

\begin{proof}
We claim that for any $A \subset \mathbb R$ such that $E^{\times}(A) \geq \frac{|A|^3}{K}$ there is a subset $A'\subseteq A$ such that
\begin{align}
\label{eq:energyLower} E^\times(A')&\gtrsim E^\times(A)\geq \frac{|A|^3}K,\\
\label{eq:dstarone}d_*(A')&\lesssim \frac{K^2|A|}{|A'|}.
\end{align}
Given such a subset $A'$, we may apply Lemma~\ref{Sh4} to find that
\begin{align*}
|A(A+\alpha)|\geq |A'(A'+\alpha)| 
&\gtrsim \frac{E^{\times}(A')^2}{|A'|^{58/13} d_*^{7/13}(A')}\\
&\gtrsim \frac{|A|^6}{K^2|A'|^{58/13}}\cdot\frac{|A'|^{7/13}}{|A|^{7/13}K^{14/13}}\\
&= \frac{|A|^{71/13}}{K^{40/13}|A'|^{51/13}}
\geq\frac{|A|^{20/13}}{K^{40/13}}.
\end{align*}

It remains to prove \eqref{eq:energyLower} and \eqref{eq:dstarone}.
By the popularity principle and dyadic pigeonholing there is a subset $P\subseteq A/A$ and a number $\Delta\geq |A|/2K$ such that for all $x$ in $P$
\[
\Delta\leq |A\cap xA| < 2\Delta,
\]
and
\[
\sum_{x\in P}|A\cap xA|^2 \gtrsim E^\times(A).
\]

Now we perform an additional refinement step.
Let $A'\subseteq A$ denote the set of $x$ such that
\[
|P\cap xA^{-1}| \geq \frac{\Delta |P|}{4|A|}.
\]
Since
\[
\sum_{x\in A}|P\cap xA^{-1}|=\sum_{x\in P}|A\cap xA|\geq \Delta|P|,
\]
by the popularity principle we have
\[
\sum_{x\in A'}|P\cap xA^{-1}| \geq \frac{3\Delta|P|}4.
\]
If $x\not\in A'$, then
\[
|P\cap x(A')^{-1}|\leq |P\cap xA^{-1}| < \frac{\Delta|P|}{4|A|}.
\]
Thus
\[
\frac{3\Delta|P|}4\leq \sum_{x\in A'}|P\cap xA^{-1}|= \sum_{x\in A}|P\cap x(A')^{-1}| \leq \frac{\Delta|P|}4 + \sum_{x\in A'}|P\cap x(A')^{-1}|,
\]
which yields
\[
\frac{\Delta|P|}2 \leq \sum_{x\in A'}|P\cap x(A')^{-1}| = \sum_{x\in P}|A'\cap xA'|.
\]
By Cauchy-Schwarz, we have
\[
E^\times(A')\gg \Delta^2|P| \gtrsim E^\times(A).
\]

Setting $Q=P$, $R=A$, and $t=(\D|P|)/(4|A|)$ in the definition of the quantity $d_{*} (A')$, we obtain
\[
d_*(A') \ll \frac{|P|^2|A|^2}{\left(\frac{\D |P|}{4|A|} \right )^3|A'|}
\ll \frac{|A|^5}{|A'| |P| \D^3} 
\lesssim \frac{|A|^5}{|A'|E^\times(A)\D}\ll \frac{K^2|A|}{|A'|},
\]
where the last inequality follows from the lower bounds for $E^\times(A)$ and $\Delta$.
\end{proof}

\begin{proof}[Proof of Theorem \ref{thm:main1}]
Write $E^{\times}(A)=\frac{|A|^3}{K}$.
By Corollary \ref{pinned} there is some $a \in A$ such that 
\begin{equation}\label{f:first_bound}
    |A(A+a)|\gg K^{1/2} |A|^{3/2} \log^{-1/2} |A| \gtrsim K^{1/2} |A|^{3/2} \,.
\end{equation}
On the other hand, for any $a \in A \setminus \{0\}$, Lemma \ref{thm:bigenergy} implies that
\begin{equation}
|A(A+a)| \gtrsim \frac{|A|^{20/13}}{K^{40/13}}.
\label{f:second_bound}
\end{equation}
Optimizing over (\ref{f:first_bound}) and (\ref{f:second_bound}), we obtain
$$
    |A(A+a)| \gtrsim |A|^{3/2 + 1/186} 
$$
as required.  

\end{proof}

In fact, by taking more care with the pigeonholing argument in the proof of Corollary \ref{pinned} it follows that the bound $|A(A+a)| \gtrsim |A|^{3/2 + 1/186} $ holds for at least half of the elements $a \in A$.

\section{Three variable expanders} \label{sec:thm2and3}

The proofs of Theorems \ref{thm:main2} and \ref{thm:main3} follow a similar argument to that of Theorem \ref{thm:main1}. We can use Corollary \ref{pinned} to get an exponent better than $3/2$ in the case when $E^{\times}(A)$ is not too large. However, in the case when $E^{\times}(A)$ is large, we need analogues of Lemma \ref{thm:bigenergy} that are quantitatively better for the purposes of these problems. These bounds are given by the following lemma.

 \begin{lemma} \label{thm:d*} Let $A \subset \mathbb R$ and suppose that $E^{\times}(A) \geq \frac{|A|^3}{K}$. Then
 \begin{equation}
 |A-A| \gtrsim \frac{|A|^{8/5}}{K^{6/5}}
 \label{A-A}
 \end{equation}
 and
  \begin{equation}
|A+A| \gtrsim \frac{|A|^{58/37}}{K^{42/37}}.
 \label{A+A}
 \end{equation}
 \end{lemma}
 
 In order to prove \eqref{A-A}, we will need the following lemma.



\begin{lemma} \label{thm:difflem} For any finite set $A \subset \mathbb R$,
$$|A-A| \gg \frac{|A|^{8/5}}{d_*^{3/5}(A)\log^{2/5}|A|}.$$
\end{lemma}

Although this result has not appeared explicitly in the literature, it can be proved by 
essentially copying the arguments from \cite{RRNS} and predecessors with the stronger Lemma
\ref{thm:KSlem} in place of the bound \eqref{KSold}. The proof is included in the appendix for completeness.

The following similar result for sum sets follows from a combination of the work in \cite{Sh2}
and \cite{KS}:
\begin{lemma} \label{thm:sh}
For any finite set $A \subset \mathbb R$,
$$|A+A| \gtrsim \frac {|A|^{58/37}}{d_*^{21/37}(A)}.$$
\end{lemma}

To be more precise, it was proven in \cite{Sh2} that if $A$ is a Szemer\'{e}di-Trotter set with parameter  $D(A)$, then
$|A+A| \gtrsim \frac {|A|^{58/37}}{D^{21/37}(A)}$, and it was subsequently established in \cite{KS} that any set $A$ is a Szemer\'{e}di-Trotter set with $O(d_*(A))$.

In addition, we need the following lemma, which uses the hypothesis that the energy is large in order to find a large subset $A' \subset A$ such that $d_*(A)$ is small.
\begin{lemma}[Double pigeonholing argument]
\label{lem:doublePigeonholing}
Let $A \subset \mathbb R$ and suppose that $E^{\times}(A) \geq \frac{|A|^3}{K}$.
Then there is a subset $A'\subseteq A$ and a number $\Delta \gg |A|/K$ such that
\begin{align}
\label{eq:aprimeLower}|A'|&\gtrsim\frac{|A|^2}{K\Delta}\\
\label{eq:dstarTwo}d_*(A')&\lesssim \frac{K|A'|^2}{|A|\Delta}.
\end{align}
\end{lemma}
The proof of Lemma~\ref{lem:doublePigeonholing} is similar to the refinement step in the proof of Lemma \ref{thm:bigenergy}.
\begin{proof}
By the popularity principle and dyadic pigeonholing there is a subset $P\subseteq A/A$ and a number $\Delta\geq |A|/2K$ such that for all $x$ in $P$
\[
\Delta\leq |A\cap xA| < 2\Delta,
\]
and
\[
\sum_{x\in P}|A\cap xA|^2 \gtrsim E^\times(A).
\]

Now we perform a second dyadic pigeonholing argument.
Note that
\[
\sum_{a\in A}|A\cap aP| = \sum_{x\in P}|A\cap xA| \geq |P|\Delta.
\]
Thus there exists a subset $A'\subseteq A$ and a number $0 < t\leq |A|$ such that for all $a$ in $A'$
\[
t\leq |A\cap aP| < 2t
\]
and
\[
\sum_{a\in A'}|A\cap aP| \gtrsim |P|\Delta,
\]
hence
\[
|A'|t\gtrsim |P|\Delta.
\]
Since $t\leq |A|$ and $|P|\Delta^2\gtrsim |A|^3/K$ we have
\[
|A'|\gtrsim\frac{|P|\Delta}{|A|} \gtrsim \frac{|A|^2}{K\Delta},
\]
which proves \eqref{eq:aprimeLower}.

For every $a\in A'$ we have $|A\cap aP|\geq t$.
Therefore we can take
\[
t=t,\quad Q=A,\quad\mbox{and}\quad R=P^{-1}
\]
in the definition of $d_*(A')$.
We then have
\begin{align*}
d_*(A') &\leq \frac{|A|^2|P|^2}{|A'|t^3}\\
&\lesssim  \frac{|A|^2|P|^2}{|A'|t^3}\cdot \left(\frac{|A'|t}{|P|\Delta} \right)^3\\
&= \frac{|A|^2|A'|^2}{|P|\Delta^3}
= \frac{|A|^2|A'|^2}{(|P|\Delta^2)\Delta}
\lesssim \frac{K|A'|^2}{|A|\Delta},
\end{align*}
which proves \eqref{eq:dstarTwo}.
\end{proof}

 \begin{proof}[Proof of Lemma \ref{thm:d*}] 
Similar to the proof of Lemma \ref{thm:bigenergy}, the idea here is to use the double pigeonholing argument (Lemma~\ref{lem:doublePigeonholing}) to find a large subset $A' \subset A$ such that $d_*(A)$ is small, and to then apply Lemmas \ref{thm:difflem} and \ref{thm:sh} to complete
the proof.

Since $E^\times(A)\geq |A|^3/K$, by Lemma~\ref{lem:doublePigeonholing} there is a subset $A'\subseteq A$ and a number $\Delta\gg |A|/K$ such that
\begin{align*}
|A'|&\gtrsim\frac{|A|^2}{K\Delta}\\
d_*(A')&\lesssim \frac{K|A'|^2}{|A|\Delta}.
\end{align*}

Applying Lemma~\ref{thm:difflem} yields
\begin{align*}
|A-A|\geq |A'-A'|
&\gtrsim \frac{|A'|^{8/5}}{d_*(A')^{3/5}}\\
&\gtrsim |A'|^{8/5}\cdot \frac{|A|^{3/5}\Delta^{3/5}}{|A'|^{6/5}K^{3/5}}
= \frac{|A'|^{2/5}|A|^{3/5}\Delta^{3/5}}{K^{3/5}}\\
&\gtrsim \frac{|A|^{4/5}}{K^{2/5}\Delta^{2/5}} \cdot \frac{|A|^{3/5}\Delta^{3/5}}{K^{3/5}}
= \frac{|A|^{7/5}}{K}\Delta^{1/5}\\
&\gg \frac{|A|^{8/5}}{K^{6/5}}.
\end{align*}

Applying Lemma~\ref{thm:sh} yields
\begin{align*}
|A+A|\geq |A'+A'|&\gtrsim \frac{|A'|^{58/37}}{d_*^{21/37}(A')}\\
&\gtrsim |A'|^{58/37}\cdot\frac{|A|^{21/37}\Delta^{21/37}}{|A'|^{42/37}K^{21/37}}
=\frac{|A'|^{16/37}|A|^{21/37}\Delta^{21/37}}{K^{21/37}}\\
&\gtrsim \frac{|A|^{32/37}}{K^{16/37}\Delta^{16/37}} \cdot \frac{|A|^{21/37}\Delta^{21/37}}{K^{21/37}}
=\frac{|A|^{53/37}}{K}\Delta^{5/37}\\
&\gg\frac{|A|^{58/37}}{K^{42/37}}.
\end{align*}
\end{proof}

We are now ready to prove the new lower bounds for $A(A-A)$ and $A(A+A)$.

 \begin{proof}[Proof of Theorem \ref{thm:main2}]
Write $E^{\times}(A)=\frac{|A|^3}{K}$.
By Corollary \ref{pinned}
\begin{equation}\label{f:first_bound1}
    |A(A+A)|\gg K^{1/2} |A|^{3/2} \log^{-1/2} |A| \gtrsim K^{1/2} |A|^{3/2} \,.
\end{equation}
On the other hand, Lemma \ref{thm:d*} implies that
\begin{equation}
|A(A+A)| \geq |A+A| \gtrsim \frac{|A|^{58/37}}{K^{42/37}}
\label{f:second_bound2}
\end{equation}
Optimizing over (\ref{f:first_bound1}) and (\ref{f:second_bound2}), we obtain
$$
    |A(A+A)| \gtrsim |A|^{3/2 + 5/242} 
$$
as required.  

\end{proof}

\begin{proof}[Proof of Theorem \ref{thm:main3}]
Write $E^{\times}(A)=\frac{|A|^3}{K}$.
By Corollary \ref{pinned}
\begin{equation}\label{f:first_bound11}
    |A(A-A)|\gg K^{1/2} |A|^{3/2} \log^{-1/2} |A| \gtrsim K^{1/2} |A|^{3/2} \,.
\end{equation}
On the other hand, Lemma \ref{thm:d*} implies that
\begin{equation}
|A(A-A)| \geq |A-A| \gtrsim \frac{|A|^{8/5}}{K^{6/5}}
\label{f:second_bound22}
\end{equation}
Optimizing over (\ref{f:first_bound11}) and (\ref{f:second_bound22}), we obtain
$$
    |A(A-A)| \gtrsim  |A|^{3/2 + 1/34} 
$$
as required.  

\end{proof}

\section{Five variable expander} \label{sec:5var}
In this section we will prove Theorem \ref{thm:main4}, based on the proof in \cite{MRNS} of the inequality
\begin{equation}
|A(A+A+A+A)| \gg \frac{ |A|^2}{\log |A|}.
\label{A(A+A+A+A)}
\end{equation}
The proof of \eqref{A(A+A+A+A)} in \cite{MRNS} follows from comparing the upper bound on the multiplicative energy in Theorem \ref{thm:soly} with the lower bound in Lemma \ref{thm:oldmainlemma}. Here, we need a suitable analogue of Lemma \ref{thm:oldmainlemma}, the proof of which relies on the following celebrated result of Guth and Katz \cite{GK}.
\begin{theorem} \label{thm:GK}
For any finite set $A \subset \mathbb R$, the number of solutions to the equation
$$(a_1-a_2)^2+(a_2-a_4)^4=(a_5-a_6)^2+(a_7-a_8)^2$$
such that $a_1,\dots,a_8 \in A$ is at most $O(|A|^6\log|A|)$.
\end{theorem}
Theorem \ref{thm:GK} is a special case of a more general geometric result which immediately implies a resolution of the Erd\H{o}s distinct distances problem up to logarithmic factors, but here it is stated only in the form in which it
will be used in this paper.

Theorem \ref{thm:GK} can be used to prove the following variation of Lemma \ref{thm:oldmainlemma}:

\begin{lemma} \label{thm:GKcor} For any finite sets $A,B \subset \mathbb R$,
$$E^+(A)|\{a+(b_1+b_2)^2 : a \in A, b_1,b_2 \in B\}|^2 \gg \frac{|A|^4|B|^2}{\log |B|}.$$
\end{lemma}

\begin{proof}
The proof proceeds by the familiar method of double counting the number of solutions
to the equation
\begin{equation}
a_1+(b_1+b_2)^2=a_2+(b_3+b_4)^2
\label{eqS}
\end{equation}
such that $a_i \in A$ and $b_i \in B$.  Let $S$ denote the number of solutions to \eqref{eqS} and write
$$A+(B+B)^2:=\{a+(b_1+b_2)^2 : a \in A, b_1,b_2 \in B \}.$$
By the Cauchy-Schwarz inequality
$$S \geq \frac{|A|^2|B|^4}{|A+(B+B)^2|}.$$
On the other hand, also by the Cauchy-Schwarz inequality,
\begin{align*}
S^2&=\left( \sum_{x} r_{A-A}(x)r_{(B+B)^2-(B+B)^2}(x) \right ) ^2
\\ & \leq \left( \sum_{x} r_{A-A}^2(x) \right ) \left(\sum_x r_{(B+B)^2-(B+B)^2}^2(x) \right )
\\&=E^+(A)  \left(\sum_x r_{(B+B)^2-(B+B)^2}^2(x) \right ).
\end{align*}
Theorem \ref{thm:GK} tells us that $ \left(\sum_x r_{(B+B)^2-(B+B)^2}^2(x) \right )=O(|B|^6\log |B|)$. Therefore,
\begin{align*}
|A|^4|B|^8 &\leq |A+(B+B)^2|^2S^2
\\ & \ll |A+(B+B)^2|^2 E^+(A) |B|^6\log |B|.
\end{align*}
After rearranging this inequality, we obtain the desired result.
\end{proof}

Unfortunately, we are not aware of a proof of Lemma \ref{thm:GKcor} which does not use the deep
results from \cite{GK}.

We are now ready to prove Theorem \ref{thm:main4}.

\begin{proof}[Proof of Theorem \ref{thm:main4}]
Apply Lemma \ref{thm:GKcor} with $A=\log A$ and $B=A+A$. We have
$$E^+(\log A)|\{(a_1+a_2+a_3+a_4)^2+\log a_5 : a_i \in A \}|^2 \gg \frac{|A|^4|A+A|^2}{\log |A|}.$$
Note that $\log a_1 + \log a_2= \log a_3 + \log a_4$ if and only if $a_1a_2=a_3a_4$, and so $E^+(\log A)=E^{\times}(A)$. We can apply Theorem \ref{thm:soly} to deduce that
$$E^+(\log A) \ll |A+A|^2 \log |A|.$$
It then follows that
$$ |\{(a_1+a_2+a_3+a_4)^2+\log a_5 : a_i \in A \}|^2 \gg \frac{|A|^4}{\log^2 |A|},$$
which completes the proof.

\end{proof}

\section*{Appendix}

The purpose of this appendix is to present a formal proof of Lemma \ref{thm:difflem}. We will call upon the following result on the relationship between different types of
energy.
\begin{lemma}[\cite{Li}, Lemma 2.4 and 2.5] \label{thm:li}
For any finite sets $A,B \subset \mathbb R$
$$|A|^2(E_{1.5}^+(A))^2 \leq (E_3^+(A))^{2/3} (E_3^+(B))^{1/3}E(A,A-B).$$
\end{lemma}
In fact, Lemma \ref{thm:li} holds for any abelian group.

\begin{proof}[Proof of Lemma \ref{thm:difflem}] 
Recall that Lemma \ref{thm:difflem} states that for any finite set $A \subset \mathbb R$,
$$|A-A| \gg \frac{|A|^{8/5}}{d_*^{3/5}(A)\log^{2/5}|A|}.$$
In order to prove this, we will first prove two energy bounds. Note that, by Lemma \ref{thm:KSlem}, 
\begin{align} \label{E3}
\begin{split}
E_3^+(A) &=  \sum_{x} r_{A-A}^3(x)
\\&=\sum_{j \geq 1}  \sum_{x:2^{j-1} \leq r_{A-A}(x) < 2^j} r_{A-A}^3(x)
\\& \ll |A|^3d_*(A) \log |A|.
\end{split}
\end{align}

Similarly, for any $F \subset \mathbb R$, and a parameter $\triangle >0$
\begin{align} \label{Emix2}
\begin{split}
E^+(A,F) &=  \sum_{x} r_{A-F}^2(x)
\\&= \sum_{x:r_{A-F}(x) \leq \triangle} r_{A-F}^2(x)+\sum_{j \geq 1}  \sum_{x:2^{j-1}\triangle \leq r_{A-A}(x) < 2^j\triangle} r_{A-F}^2(x)
\\& \ll \triangle|A||F|+ \frac{|A||F|^2d_*(A)}{\triangle}.
\end{split}
\end{align}
We choose $\triangle = (|F|d_*(A))^{1/2}$, and thus conclude that
\begin{equation}
E(A,F) \ll |A||F|^{3/2}d_*(A)^{1/2}.
\label{Emix}
\end{equation}

Now, by H\"{o}lder's inequality
\begin{align} \label{E1.5}
\begin{split}
|A|^6&=  \left ( \sum_{x \in A-A} r_{A-A}(x) \right )^3
\\& \leq  |A-A|\left (\sum_{x} r_{A-A}^{3/2}(x) \right)^2
\\& = |A-A|(E_{1.5}^+(A))^2.
\end{split}
\end{align}
Applying \eqref{E1.5} with Lemma \ref{thm:li}, as well as inequalities \eqref{E3} and \eqref{Emix}, we have
\begin{align*}
|A|^8& \leq    |A-A|(E_{1.5}^+(A))^2 |A|^2
\\& \leq  |A-A|E_3^+(A)E(A,A-A)
\\& = |A-A|^{5/2}|A|^4d_*(A)^{3/2}\log |A|.
\end{align*}
Rearranging this inequality completes the proof.
\end{proof} 

Finally, we note that a similar method can be used to prove a quantitatively weaker version of Lemma \ref{thm:sh}, in the form of the following result:
\begin{equation}
|A+A| \gg \frac{|A|^{14/9}}{d_*^{5/9}(A)\log^{2/9}|A|}.
\label{weaker}
\end{equation}
To see how this works, one can repeat
the arguments from the proof of Theorem 1.2 in \cite{LRN}, but using Lemma \ref{thm:KSlem} in place of Lemma 3.2 from \cite{LRN}. This is worth noting, since the proofs of the main results in \cite{KS} and \cite{KS2}, that is
the bound 
$$\max \{|A+A|,|AA|\} \gg |A|^{4/3+c}$$
for some $c>0$, both include applications of Lemma \ref{thm:sh}. One can also obtain this sum-product
estimate, albeit with a smaller positive value $c$, by using the bound \eqref{weaker} instead of
Lemma \ref{thm:sh}.

\section*{Acknowledgement}


The second author is grateful for the hospitality of the R\'{e}nyi Institute, where part of this work was conducted with support from Grant ERC-AdG. 321104, and also for the support of the Austrian Science Fund FWF Project F5511-N26,
which is part of the Special Research Program "Quasi-Monte Carlo Methods: Theory and Applications". 
We thank Antal Balog, Orit Raz and Endre Szemer\'{e}di for helpful discussions.

\end{document}